\newif\ifpictures
\picturestrue

\documentclass[12pt]{amsart}
\usepackage{amssymb}
\usepackage{amsmath}
\usepackage{amscd}
\usepackage[pdftex]{graphicx}
\usepackage{hyperref}

\numberwithin{equation}{section}
\newtheorem{thm}{Theorem}

\newtheorem{lemma}[thm]{Lemma}
\newtheorem{cor}[thm]{Corollary}

\theoremstyle{definition}

\newtheorem{example}[thm]{Example}
\newtheorem{remark1}[thm]{Remark}
\newtheorem{openproblem1}[thm]{Open problem}
\newtheorem{definition}[thm]{Definition}

\numberwithin{thm}{section}

\newcounter{FNC}[page]
\def\newfootnote#1{{\addtocounter{FNC}{2}$^\fnsymbol{FNC}$%
     \let\thefootnote\relax\footnotetext{$^\fnsymbol{FNC}$#1}}}

\newcommand{\C}{\mathbb{C}}
\newcommand{\N}{\mathbb{N}}

\newcommand{\R}{\mathbb{R}}

\renewcommand{\P}{\mathbb{P}}
\newcommand{\Z}{\mathbb{Z}}

\newcommand{\T}{\mathbb{T}}

\newcommand{\Sph}{\mathbb{S}}

\newcommand{\compl}{\mathsf{c}}

\DeclareMathOperator{\New}{New}
\DeclareMathOperator{\Res}{Res} 
\DeclareMathOperator{\Disc}{Disc} 
\DeclareMathOperator{\vol}{vol} 

\let\Re\relax
\DeclareMathOperator{\Re}{Re}
\let\Im\relax
\DeclareMathOperator{\Im}{Im}

\usepackage{xcolor}

\title[]{Imaginary projections of polynomials}

\author{Thorsten J\"orgens}

\author{Thorsten Theobald}

\author{Timo de Wolff}

\address{Thorsten J\"orgens and Thorsten Theobald, 
  Goethe-Universit\"at, FB 12 -- Institut f\"ur Mathematik,
  Postfach 11 19 32, D-60054 Frankfurt am Main, Germany
}
\email{joergens@math.uni-frankfurt.de, theobald@math.uni-frankfurt.de}
\address{Timo de Wolff, TU Berlin, Stra{\ss}e des 17.~Juni 136, 10623 Berlin, Germany}
\email{dewolff@math.tu-berlin.de}

\date{\today}
\subjclass[2000]{14P10, 12D10}
\keywords{Imaginary projection, stable polynomial, convex algebraic geometry, amoeba,
  component of the complement} 
\begin{document}

\begin{abstract}
We introduce the imaginary projection of a multivariate polynomial
$f \in \C[\mathbf{z}]$ as the projection of the variety of $f$
onto its imaginary part,
$\mathcal{I}(f) \ = \ \{\Im(\mathbf{z}) \, : \, \mathbf{z} \in \mathcal{V}(f) \}$. 
Since a polynomial $f$ is stable if and only if 
$\mathcal{I}(f) \cap \R_{>0}^n \ = \ \emptyset$, the notion offers a
novel geometric view underlying stability questions of polynomials.

We show that the connected components of the complement of the closure of the imaginary projections 
are convex,  thus opening a central connection to the theory of amoebas and coamoebas.
Building upon this, the paper establishes structural properties
of the components of the complement, such as lower bounds on their maximal number,
proves a complete classification of the imaginary projections of quadratic polynomials
and characterizes the limit directions for polynomials of arbitrary degree.
\end{abstract}

\maketitle

\section{Introduction}

Recent years have seen a lot of interest in stable polynomials, see, e.g., 
\cite{borcea-braenden-2008, borcea-braenden-2009,
mss-interlacing2, wagner-2011} and the references therein.
A polynomial $f = f(\mathbf{z}) = f(z_1, \ldots, z_n) \in \C[\mathbf{z}] = \C[z_1,\ldots,z_n]$ is called \emph{stable} if every root $\mathbf{z}$ satisfies $\Im(z_j) \leq 0$ for some $j$.
We call $f$ \emph{real stable} if $f$ has real coefficients and
is stable.

As recent prominent applications,
Marcus, Spielman, and Srivastava employed stable polynomials in the proof of the
Kadison-Singer Conjecture \cite{mss-interlacing2} and in the existence proof of families of
bipartite Ramanujan graphs of every degree larger than two~\cite{mss-interlacing1}. Stable polynomials 
have also been used by Borcea and Br\"and\'{e}n to prove Johnson's 
Conjecture \cite{borcea-braenden-2008} and in Gurvits' simple proof of a generalization 
of van der Waerden's Conjecture for permanents \cite{gurvits-2008}.
Moreover, there are strong connections to hyperbolic polynomials and their hyperbolicity cones, see Section~\ref{se:stable}.

In this paper, we 
initiate to study the underlying projections on the imaginary 
parts from a geometric point of view. Given a polynomial
$f  \in \C[\mathbf{z}]$, introduce the \emph{imaginary projection} of $f$ as
\[
  \mathcal{I}(f) \ = \ \left\{\Im(\mathbf{z}) \, : \, \mathbf{z} \in \mathcal{V}(f) \right\} \ \subseteq \ \R^n \, ,
\]
where $\mathcal{V}(f)$ denotes the variety of $f$ and $\Im(\mathbf{z})=(\Im(z_1),\ldots,\Im(z_n))$. So, in particular,
$f$ is stable if and only if 
\[
\mathcal{I}(f) \cap \R_{>0}^n \ = \ \emptyset \, .
\]

Our work is motivated by the theory of \textit{amoebas} as well as by the
general goal to reveal and understand convexity phenomena in 
algebraic geometry, see \cite{bpt-2013}.
Amoebas are the images of algebraic varieties in the algebraic torus $(\C^*)^n$ under the log-absolute map:
\[
  \mathcal{A}(f) \ = \ \{(\log|z_1|, \ldots, \log|z_n|) \, : \,
  \mathbf{z} \in \mathcal{V}(f) \cap (\C^*)^n \} \ \subseteq \ \R^n \, ,
\]
see \cite{gelfand-1994}. Coamoebas employ the arg-map 
rather than the log-absolute map; see, e.g., \cite{forsgard-diss}.

For amoebas, important structural results as well as their occurrences in
a broad spectrum of mathematical disciplines
have been intensively studied, see
\cite{mikhalkin-different-faces,passare-rullgard-2004,passare-tsikh-2005} as well
as the recent survey \cite{dewolff-amoeba-survey}.
For coamoebas, investigations are much more recent
\cite{forsgard-diss,forsgard-johansson-2015,nisse-sottile-2013}.
A prominent result states that the complement of an amoeba as well as
the complement of the closure of a coamoeba consists of finitely many convex components, see
\cite{forsgard-johansson-2015,gelfand-1994}. As a key result, which
also motivates our study, we show that the closure of the complement of the imaginary 
projection of a polynomial consists of finitely many
convex components as well, see Theorem~\ref{th:convex}.

While there are important analogies among amoebas, coamoebas, and
imaginary projections, there are also fundamental differences between 
these structures. The fibers of the log-absolute
maps underlying amoebas are compact, whereas for imaginary
projections they are not compact. Furthermore, the limit directions 
of amoebas, also known as tentacles, are characterized by the logarithmic limit sets and
thus carry a polyhedral structure; see 
\cite[Theorem 1.4.2]{maclagan-sturmfels-book}. In contrast,
the limit directions of the imaginary projections are not polyhedral in general,
see Section~\ref{sec:infinity}. For coamoebas, which are defined on a 
torus,  Nisse and Sottile have introduced a variant of the logarithmic limit sets,
by considering accumulation points of arguments of sequences with unbounded 
logarithm \cite{nisse-sottile-2013}.

Building upon the fundamental convexity result,
we study structural properties of imaginary projections.
We also give lower bounds on
the maximal number of components of the complement, see Corollary~\ref{co:lowerbound}.

We investigate important subclasses, such as quadratic and
multilinear polynomials. For the class of real quadratic polynomials,
we can provide a complete classification of the
imaginary projections,  see Theorem~\ref{th:quadrics}.
Indeed, this
classification result in Theorem~\ref{th:quadrics} is somewhat unexpected,
since it involves various qualitatively different cases.

Starting from the well-known results
on tentacles of amoebas, we characterize the limit points of the imaginary
projections. Contrary to the case of the amoeba of a non-zero polynomial $f$,
it is possible that every point on the sphere $\Sph^{n-1}$ is a limit direction of
the imaginary projection of $f$.
For $f \in \C[\mathbf{z}]$, we provide a criterion for one-dimensional families of limit
directions at infinity.
In the case $n=2$ this also characterizes the situations
that all points are limit points. See Theorem~\ref{thm:limit-directions-general-case}
and Corollary~\ref{co:bivar-limit} for further details.

It is easy to see that real projections of complex polynomials should
behave in the same way as imaginary projections, 
since one projection is easily
seen to be an instance of the other by replacing the polynomial $f(x)$
with $f(\sqrt{-1} x)$. However, we focus on imaginary projections since
the latter projections are more naturally connected to stability and
hyperbolicity in the setting of real polynomials.

\smallskip

Our paper is structured as follows. Section~\ref{se:prelim} collects 
existing facts on stable polynomials as well as on amoebas and coamoebas. In 
Section~\ref{se:structure}, we study the structure of imaginary projections.
Section~\ref{sec:complcomponents} considers the components of the complement.
In Section~\ref{se:quadratic}, we discuss quadratic and multilinear (in the
sense of multi-affine-linear) polynomials, and 
Section~\ref{sec:infinity} is concerned with
the situation at infinity. In Section~\ref{se:openquestions} we close with some
open questions.

\section{Preliminaries\label{se:prelim}}

We collect basic notions on stable polynomials as well as on amoebas and coamoebas. 
Let $\R_{\ge 0}$ and $\R_{>0}$ denote the set of non-negative and the set of strictly
positive real numbers.

Throughout the paper, we use bold letters for vectors, e.g., $\mathbf{z}=(z_1,\ldots,z_n)\in\C^n$. Unless stated otherwise, the dimension of these vectors is $n$.  Denote by $\mathcal{V}(f)$
the complex variety of a polynomial $f \in \C[\mathbf{z}]$ and by $\mathcal{V}_{\R}(f)$ the real variety of $f$.

We denote by $\Re(z)$ and $\Im(z)$ the real and the imaginary part of a point $z\in\C$, i.e., 
$z=\Re(z) +i \Im(z)$, 
and component-wise for points $\mathbf{z}\in\C^n$. 
For an arbitrary set $M\subseteq\C^n$ we understand $\Re(M)$ and $\Im(M)$ 
as the real parts and the imaginary parts of all elements in $M$. Moreover, for a polynomial $f\in\C[\mathbf{z}]$ we denote by $\Re(f)$ and $\Im(f)$ the real part and the imaginary part of $f$ after the realification $\mathbf{z}=\mathbf{x}+i\mathbf{y}\in\C^n\mapsto(\mathbf{x},\mathbf{y})\in\R^{2n}$, i.e., $f(\mathbf{x},\mathbf{y})=\Re f(\mathbf{x},\mathbf{y})+i\Im f(\mathbf{x},\mathbf{y})$. Note that $\Re(f)$ and $\Im(f)$ are real polynomials in $\R[\mathbf{x},\mathbf{y}]$.

Furthermore, we use the notations $\mathcal{H}_\C^n$ for the set $\{\mathbf{z}\in\C^n: \Im(z_j)>0, \, 1 \le j \le n\}$ and $\mathcal{H}_\R^n=\Im\mathcal{H}_\C^n$, which is the positive orthant. 

\subsection{Stable polynomials\label{se:stable}}

Based on the notions of stability and real stability defined in the introduction,
we collect the following statements and properties. As a general source
on stability of polynomials, we refer to~\cite{wagner-2011} and the references therein.

\begin{definition}\label{def:stability}
	A polynomial $f \in \C[\mathbf{z}]$ is called \emph{stable}
	if it has no root $\mathbf{z}$ in $\mathcal{H}_\C^n$.
\end{definition}

\begin{example}\cite[Proposition 2.4]{borcea-braenden-2008}
For positive semidefinite $d\times d$-matrices $A_1,\ldots,A_n$ and a Hermitian 
$d\times d$-matrix $B$, the polynomial
\begin{equation*}
f(\mathbf{z}) \ = \ \det(z_1A_1+\cdots+z_nA_n+B)
\end{equation*}
is real stable or identically zero.
\end{example}

There is a close connection between stable, homogeneous polynomials and hyperbolic polynomials. A homogeneous polynomial $f\in\R[\mathbf{z}]$ is called 
\emph{hyperbolic} 
in direction $\mathbf{e}\in\R^n$, if $f(\mathbf{e})\neq0$ and for every $\mathbf{x} \in \R^n$ the real function
$t\mapsto f(\mathbf{x}+t\mathbf{e})$ has only real roots. It is known that a
homogeneous polynomial $f\in\R[\mathbf{z}]$ is real stable if and only if $f$ is hyperbolic with respect to every point in the positive orthant, see \cite{garding-59, wagner-2011}. 
 
Stability of univariate polynomials can be tested as follows.
Here, we call two univariate polynomials $f, g \in \R[\mathbf{z}]$ in \emph{proper position}, $f\ll g$, if the zeros of $f$ and $g$ \emph{interlace} (i.e., alternate, see 
\cite{braenden2007,mss-interlacing1}),
and if their \textit{Wronskian} $W[f,g]:=f'g-fg'$ is non-negative on $\R$. 
Note that if the roots of two polynomials $f$ and $g$ interlace, then $W[f,g]$ is non-negative or 
non-positive.

\begin{thm}{\normalfont (Hermite-Biehler, see \cite[Thm.~6.3.4]{rahman} or \cite{wagner-2011})} A non-constant univariate polynomial $f \in \C[z]$ is stable if and only if $\Im f \ll \Re f$.
\end{thm}

Borcea and Br\"and\'{e}n gave the following multivariate generalization of the Hermite-Biehler Theorem. 
Here, two multivariate polynomials $f,g\in\R[\mathbf{z}]$ are called in 
\emph{proper position}, written $f\ll g$, if the univariate polynomials $f(\mathbf{x}+t\mathbf{e})$, $g(\mathbf{x}+t\mathbf{e})$ are in proper position 
for all $\mathbf{x}\in\R^n$, $\mathbf{e}\in\R_{\ge 0}^n \setminus \{0\}$.

\begin{thm}{\normalfont (\cite[Cor.~2.4]{borcea-braenden-2010}, see also \cite[Thm.~5.3]{braenden2007})}
A non-constant polynomial  $f \in\C[\mathbf{z}]$ is stable if and only if $\Im f  \ll \Re f$.
\end{thm}

We call a polynomial $f \in \C[\mathbf{z}]$ \emph{multilinear}
if it has degree at most $1$ with respect to each variable.
Br\"and{\'e}n characterized stability of multilinear
polynomials with real coefficients.

\begin{thm}\cite[Theorem 5.6]{braenden2007} \label{th:braenden1}
Let $f\in\R[\mathbf{z}]$ be non-constant and multilinear. Then $f$ is stable if and only
if for all $1\leq j,k\leq n$ the function
\begin{equation*}
\Delta_{jk}(f) \ = \ \frac{\partial f}{\partial z_j}\frac{\partial f}{\partial z_k}  -\frac{\partial^2 f}{\partial z_j\partial z_k} \cdot f
\end{equation*}
is non-negative on $\R^n$.
\end{thm}

Hence, a non-zero bivariate polynomial $f(z_1,z_2)=\alpha z_1z_2+\beta z_1+\gamma z_2+\delta$ with $\alpha, \beta, \gamma, \delta\in\R$ is real stable if and only if 
$\beta\gamma-\alpha\delta\geq0$.

The non-multilinear case can be reduced to the multilinear case via the \emph{polarization} $\mathcal{P}(f)$ of a multivariate polynomial $f$, see \cite{braenden2007}. Denoting by $d_j$ the degree of $f$ in
the variable $z_j$, $\mathcal{P}(f)$ is the unique polynomial in the 
variables $z_{jk}$, $1\leq j\leq n$, $1\leq k\leq d_j$ with the properties
\begin{enumerate}
\item $\mathcal{P}(f)$ is multilinear,
\item $\mathcal{P}(f)$ is symmetric in the variables $z_{j1},\ldots,z_{jd_j}$, $1\leq j\leq n$,
\item if we apply the substitutions $z_{jk}=z_j$ for all $j,k$, then $\mathcal{P}(f)$ coincides with $f$.
\end{enumerate}
By the Grace-Walsh-Szeg\"{o} Theorem, 
$\mathcal{P}(f)$ is stable if and only if $f$ is stable; see, e.g., \cite[Cor.\ 5.9]{braenden2007}.

\medskip

By Theorem~\ref{th:braenden1}, deciding whether a multilinear polynomial $f$ is 
stable is equivalent to deciding whether $\Delta_{jk}(f)\geq0$ on $\R^n$ for all $j,k$. In~\cite{kpv-2012},
sum of squares-relaxations are considered to decide this question.

\subsection{Amoebas and coamoebas\label{se:amoebascoamoebas}}

The theory of amoebas builds upon algebraic varieties in the complex torus $(\C^\ast)^n = (\C\setminus\{0\})^n$.
For a Laurent polynomial $f \in \C[z_1^{\pm 1}, \ldots, z_n^{\pm 1}]$, define the
\emph{semialgebraic amoeba} $\mathcal{S}(f)$
(also known as \emph{unlog amoeba}) by
\begin{equation}
	\mathcal{S}(f) \ = \ \{|\mathbf{z}|=(|z_1|,\ldots,|z_n|)\in\R^n:\mathbf{z}\in
         \mathcal{V}(f) \cap (\C^*)^n\},
\end{equation}
and the \emph{amoeba} $\mathcal{A}(f)$ by
\begin{equation}
	\mathcal{A}(f) \ = \ \{\log|\mathbf{z}|=(\log|z_1|,\ldots,\log|z_n|)\in\R^n:\mathbf{z}\in\mathcal{V}(f)\cap(\C^\ast)^n\} \, .
\end{equation}
 
Amoebas were first 
introduced and studied by Gelfand, Kapranov and Zelevinsky in 
\cite{gelfand-1994}. Similarly, the \emph{coamoeba} of $f$ is defined as
\begin{equation}
	co\mathcal{A}(f) \ = \ \{\arg(\mathbf{z})=(\arg(z_1),\ldots,\arg(z_n)):\mathbf{z}\in\mathcal{V}(f)\cap(\C^\ast)^n\} \ \subseteq \ \T^n \, ,
\end{equation}
where $\arg$ denotes the argument of a complex number and
$\T^n =\left(\R/2\pi\Z\right)^n$.

If $\log_{\C}$ is the complex logarithm, then we have the relations
\[
  \mathcal{A}(f) = \Re\circ\log_{\C}\mathcal{V}(f) \quad \text{ and } \quad
  co\mathcal{A}(f) =\Im\circ\log_{\C}\mathcal{V}(f) \, ,
\]
where all maps are understood component-wise. See Figure \ref{Fig:AmoebaCoamoeba} for an example of an amoeba and a coamoeba.

We recall some basic statements about amoebas, see \cite{fpt-2000,gelfand-1994,genus1}.
For a Laurent polynomial $f \in \C[\mathbf{z}]$,
the amoeba $\mathcal{A}(f)$ is a closed set. The complement
of $\mathcal{A}(f)$
consists of finitely many convex regions, and
these regions are in bijective correspondence with the different
Laurent series expansions of the rational function $1/f$.
The number of components in the complement of an amoeba is bounded from above by the 
number of lattice points in the Newton polytope of $f$ and bounded from below by the number of vertices of the Newton polytope of $f$.

\begin{figure}[ht]
\ifpictures
\includegraphics[height=0.26\linewidth]{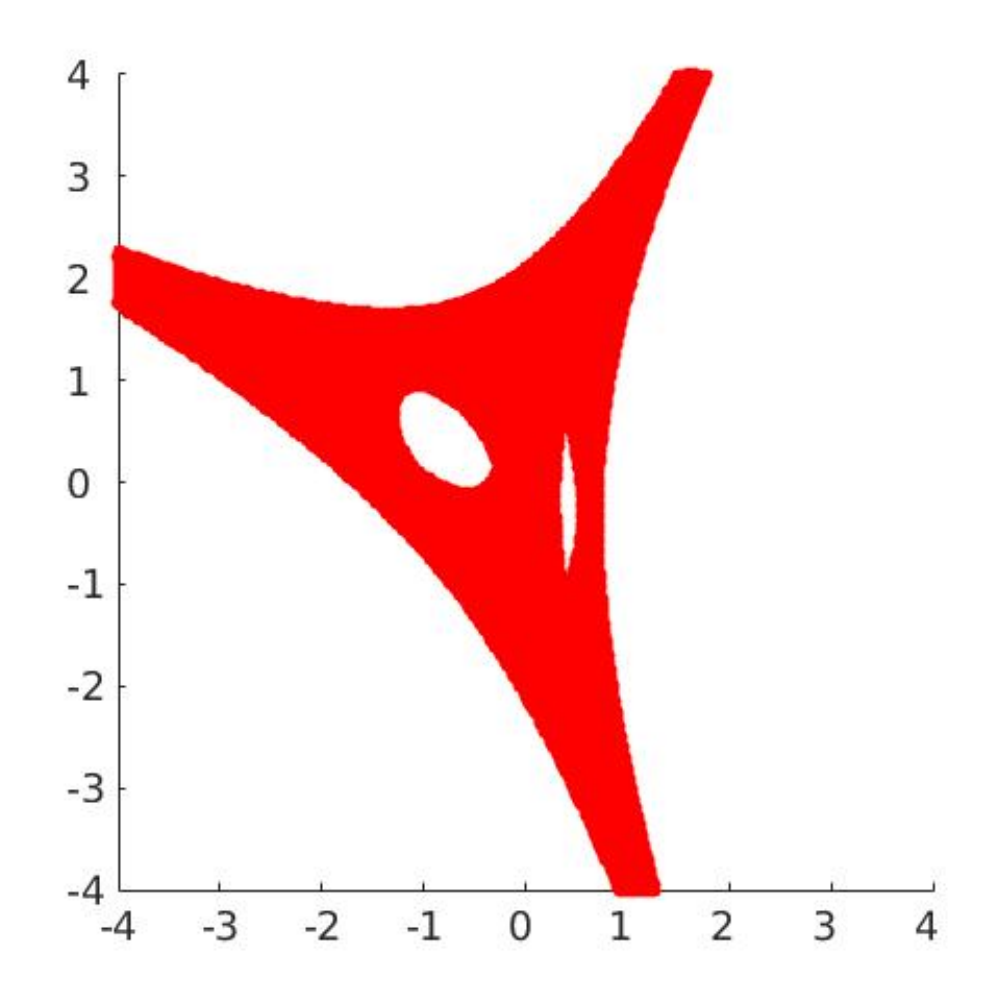} \quad
\includegraphics[height=0.26\linewidth]{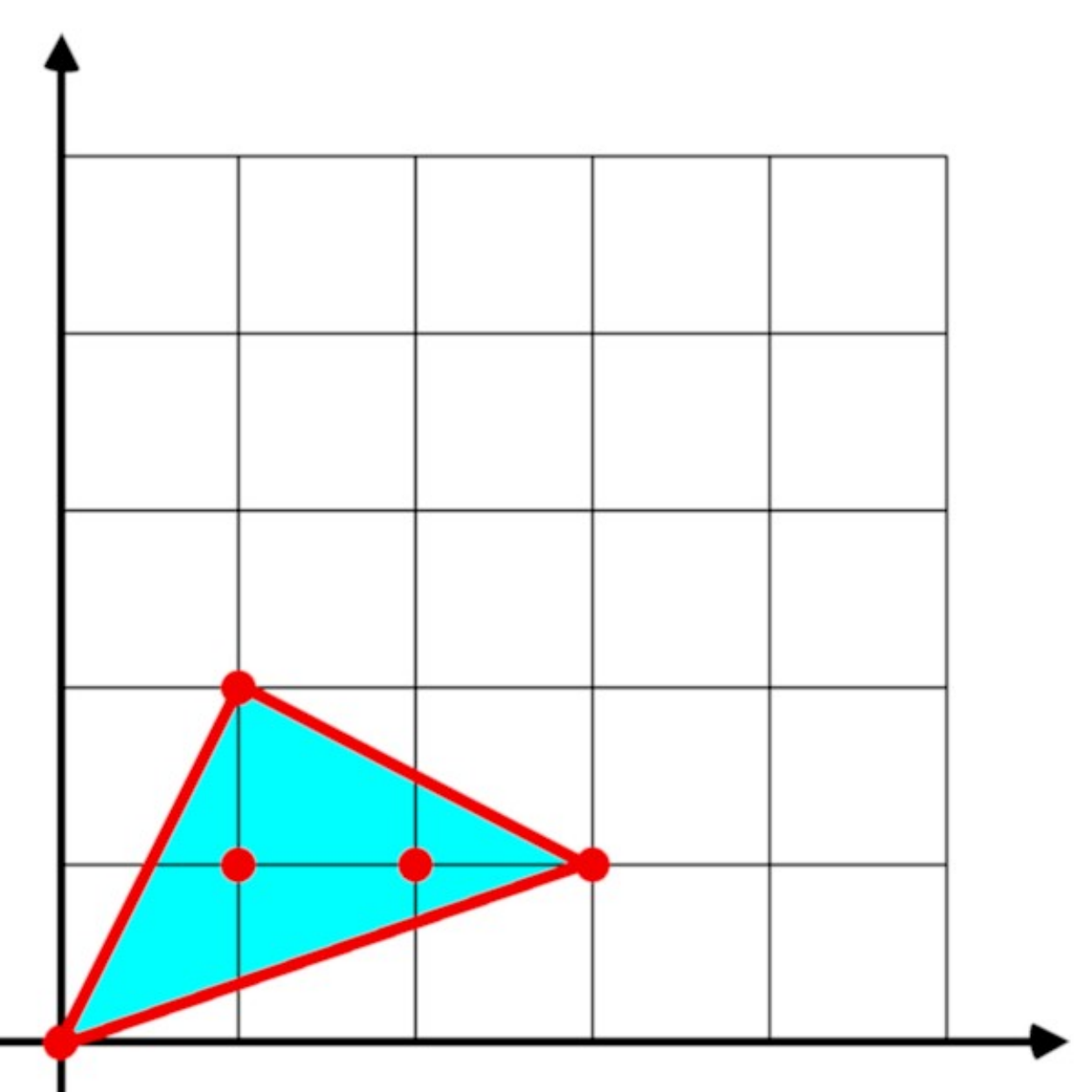}
\includegraphics[height=0.26\linewidth]{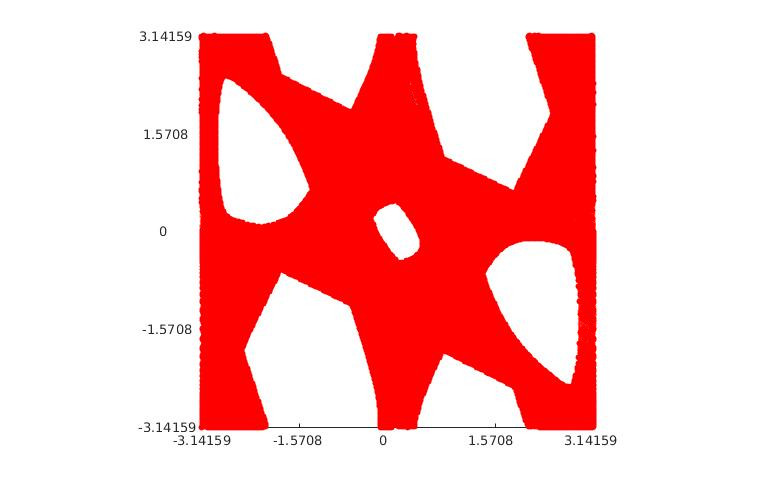}

\fi
\caption{An approximation of the amoeba and the coamoeba of the Laurent polynomial $f(z_1,z_2) = 2z_1^3z_2 + z_1z_2^2 - 4z_1z_2 - 2.5 \cdot e^{0.7 \cdot \pi \cdot i} z_1^2z_2 + 1$ together with its corresponding Newton polytope.}
\label{Fig:AmoebaCoamoeba}
\end{figure}

For coamoebas, it has been conjectured that
the complement of the closure of $co\mathcal{A}(f)$ contains
at most $n!\vol\New(f)$ connected components, 
where $\vol$ denotes the volume, see \cite{forsgard-diss}
for more background as well as a proof for the special case
$n=2$.
One can also consider amoebas and coamoebas of arbitrary varieties
rather than of hypersurfaces alone, see, e.g., \cite{theobald-de-wolff-approximating}.

\section{The structure of the imaginary projection of polynomials\label{se:structure}}

We investigate the structure of the imaginary projection of multivariate 
polynomials. 
	Writing $z_j = x_j + i y_j$ with real variables $x_j$, $y_j$, we see that
	$\mathcal{I}(f)$ is the projection 
	\begin{equation}
        \label{eq:improj}
	\R^{2n} \to \R^n \, , \quad (x_1, y_1, \ldots, x_n, y_n) \mapsto (y_1, \ldots, y_n)
	\end{equation}
	of a real algebraic variety, and thus $\mathcal{I}(f)$ is a semialgebraic set. Since the map~\eqref{eq:improj} is continuous, 
the imaginary projection of an irreducible polynomial $f$ is connected. See 
Figure~\ref{fi:ellipse} for an example.

\begin{figure}[ht]
	\includegraphics[width=0.26\linewidth]{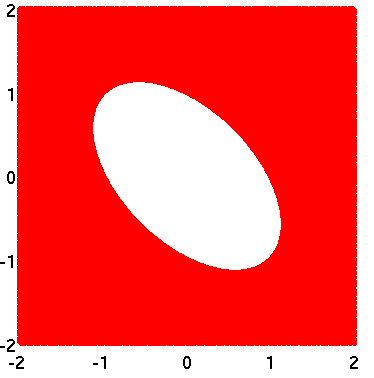}
	\caption{The imaginary projection of $f(z_1,z_2)=z_1^2 + z_2^2 +z_1z_2 + z_1 + z_2 + 1$, intersected with $[-2,2]^2$.}
      \label{fi:ellipse}
\end{figure}

The following fact allows to reduce the case of reducible polynomials to the case of
irreducible polynomials.

\begin{lemma}\label{lemma:imaginary-projection-product}
  For $f_1,f_2 \in \C[\mathbf{z}]$, we have
  $\mathcal{I}(f_1\cdot f_2)=\mathcal{I}(f_1)\cup\mathcal{I}(f_2)$.
\end{lemma}
\begin{proof} By definition of the imaginary projection we have
	\[
	\mathcal{I}(f_1\cdot f_2) \, = \, \Im\mathcal{V}(f_1\cdot f_2) \, = \, \Im(\mathcal{V}(f_1)\cup\mathcal{V}(f_2)) \, = \, \Im\mathcal{V}(f_1)\cup\Im\mathcal{V}(f_2) \, = \, \mathcal{I}(f_1)\cup\mathcal{I}(f_2).
	\]
\end{proof}

The imaginary projections of affine hyperplanes can be characterized as follows.

\begin{thm}\label{thm:linear-case-imaginary-projection}
For every affine-linear polynomial $f = a_0 + \sum_{j=1}^n a_j z_j \in \C[\mathbf{z}]$
with $(a_1, \ldots, a_n) \neq {\bf 0}$ the following statements hold.
\begin{enumerate}
 \item $\mathcal{I}(f) = \begin{cases}
  \mathcal{V}_{\R}\left( \Im(a_0 e^{-i \varphi} \right) + \sum\limits_{j=1}^n a_j e^{-i\varphi} y_j) \, , & \text{if } a_0 \in \C \text{ and } (a_1, \ldots, a_n) = \mathbf{b} \cdot e^{i \varphi} \\
  & \quad \text{ with some } \mathbf{b} \in \R^n 
  \text{ and } \varphi \in [0, 2 \pi) \, . \\
  \R^n \, , & \text{otherwise.\raisebox{4ex}{\hspace*{0.1cm}}}
\end{cases}$ \smallskip

 \item If all coefficients of $f$ are real, then $f$ is stable if and only if
  $a_1, \ldots, a_n \ge 0$ or $a_1, \ldots, a_n \le 0$.
\end{enumerate}
\end{thm}

Note that by statement (1), an affine-linear polynomial $f$ cannot be stable if 
$(a_1, \ldots, a_n) \not \in e^{i \varphi} \cdot \R^n$, and thus statement (2) provides a
complete classification for the stability of an affine-linear polynomial.

\begin{proof}
If all coefficients $a_1, \ldots, a_n$ are real, then 
\begin{eqnarray*}
  \mathcal{I}(f) & = & \left\{ \mathbf{y} \in \R^n \, : \, \exists \mathbf{x} \in \R^n \; \Re(a_0) + \sum_{j=1}^n a_j x_j = 0 \text{ and }
  \Im(a_0) + \sum_{j=1}^n a_j y_j = 0 \right\} \\
  & = &  \left\{ \mathbf{y} \in \R^n \, : \, \Im(a_0) + \sum_{j=1}^n a_j y_j = 0 \right\} \, ,
\end{eqnarray*}
and in the situation $(a_1, \ldots, a_n) \in e^{i \varphi} \cdot \R^n$, apply the real case
to $e^{-i \varphi} f$.

Now assume that $(a_1, \ldots, a_n)$ is not a complex multiple of a real vector.
That is, the real matrix  
$\begin{pmatrix} \Re(a_1) & \cdots & \Re(a_n) \\ \Im(a_1) & \cdots & \Im(a_ n)\end{pmatrix}$
has rank~2. By possibly changing the order of the coefficients $a_j$, we can assume 
that the matrix
$A = \begin{pmatrix} \Re(a_1) & \Re(a_2) \\ \Im(a_1) & \Im(a_2) \end{pmatrix}$
is invertible. In order to show $\mathcal{I}(f) = \R^n$, 
consider a fixed $\mathbf{y} \in \R^n$ and choose arbitrary 
$x_3, \ldots, x_n \in \R$. Then the conditions 
$\Re f(\mathbf{x}+i\mathbf{y}) = 0$ and 
$\Im f(\mathbf{x}+i\mathbf{y}) = 0$ yield a system of two
real linear equations in $x_1, x_2$ with coefficient matrix $A$,
\begin{eqnarray*}
0 & = & \Re f(\mathbf{x}+i\mathbf{y}) \ = \ \Re a_0 + \sum_{j=1}^n \Re(a_j) x_j -
  \sum_{j=1}^n \Im(a_j) y_j \, , \\
0 & = & \Im f(\mathbf{x}+i\mathbf{y}) \ = \ \Im a_0 + \sum_{j=1}^n \Im(a_j) x_j +
  \sum_{j=1}^n \Re(a_j) y_j \, .
\end{eqnarray*}
Since $A$ is invertible and $x_3, \ldots, x_n$ are fixed, there 
exists a solution $x_1,x_2 \in \R$, and thus $\mathbf{y} \in \mathcal{I}(f)$. 
This completes the proof of (1). 

Now let all coefficients of $f$ be real. By part~(1), $f$ has a zero with $\Im(z_j) > 0$ for
all $j$ if and only if there exists at least one positive coefficient and one negative coefficient.
\end{proof}

\begin{cor} Let $n \ge 2$ and $f = a_0 + \sum_{j=1}^n a_j z_j$ be a stable affine-linear polynomial. 
Then there exists a 
 (complex) $\varepsilon$-perturbation
 of the coefficients such that the resulting polynomial is not stable. 
 If all coefficients $a_0, \ldots, a_n$ are real and non-zero then
 for any real infinitesimal perturbations the stability of $f$ is preserved.
\end{cor}

\begin{figure}[ht]
  \[
      \includegraphics[width=4cm]{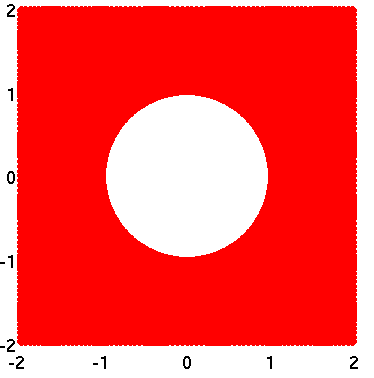} \qquad \qquad
      \includegraphics[width=4cm]{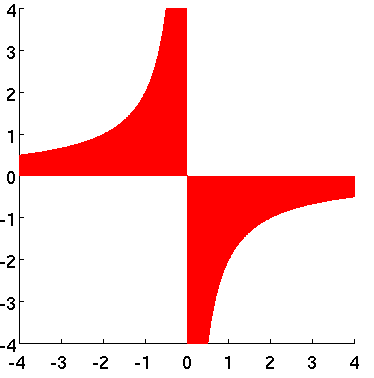}
  \]
  \caption{The imaginary projections of 
    $f(z_1,z_2)=z_1^2 + z_2^2 + 1$ and $f(z_1,z_2)=z_1z_2+z_1+z_2-1$.}
  \label{fi:top1}
\end{figure}

The set $\mathcal{I}(f)$ is not always a closed set. Indeed, already in the quadratic setting
all the following cases can occur.
\begin{enumerate}
	\item $\mathcal{I}(f)$ is open for $f(z_1,z_2)=z_1^2+z_2^2-1$. In fact, $\mathcal{I}(f) = \R^n$.
	\item $\mathcal{I}(f)$ is closed for $f(z_1,z_2)=z_1^2+z_2^2+1$.
	\item $\mathcal{I}(f)$ is neither open nor closed for $f(z_1,z_2)=z_1z_2+z_1+z_2-1$.
	The hyperbolic curve belongs to $\mathcal{I}(f)$, but, except the origin, the axes do not
      belong to $\mathcal{I}(f)$.
\end{enumerate}
See Figure~\ref{fi:top1}, and for further details on the specific examples we refer to 
the discussion of quadratic polynomials in Section~\ref{se:quadratic}.

\begin{openproblem1}\label{thm:imaginary-projection-is-never-open}
	Let $f\in\C[\mathbf{z}]$ be a polynomial. Is $\mathcal{I}(f)$ open if and only if $\mathcal{I}(f)=\R^n$? Clearly, the if-direction is obvious.
\end{openproblem1}

\begin{remark1} \label{re:symm}
If $f$ has real coefficients, then the zeros of $f$ come in
conjugated pairs. Therefore, $\mathcal{I}(f)$ is symmetric with respect 
to the origin.
\end{remark1}

\section{Components of the complement\label{sec:complcomponents}}

Similar to amoebas and coamoebas, the complement of an imaginary projection can have several 
connected components. In contrast to amoebas, already quadratic polynomials can lead
to bounded components in the complement. Indeed, the complement  of the imaginary projection of 
$f(z_1,z_2) = z_1^2 + z_2^2 + 1$ has a bounded component, see
Figure~\ref{fi:top1}. The existence of this bounded component of the complement is a consequence of
$\Re(f(x_1+iy_1,x_2+iy_2))=x_1^2-y_1^2+x_2^2-y_2^2+1$. If $y_1^2+y_2^2<1$, then
there cannot be any $x_1,x_2\in\R$ with $\Re(f)=0$. Also note that the origin is contained in the complement of an imaginary projection $\mathcal{I}(f)$ 
whenever $f$ has a real solution.

Set $A^{\compl} = \R^n \setminus A$ for the complement of a set $A \subset \R^n$,
and write $\overline{A}$ for the closure of~$A$.
As pointed out in Section~\ref{se:amoebascoamoebas}, it is an important property of amoebas and coamoebas that the 
components of $\mathcal{A}(f)$ and of $\overline{co\mathcal{A}(f)}$ are convex.
As a key property of imaginary projections, we show that the closure their 
complement consists of convex components as well.

\begin{thm} \label{th:convex}
For every polynomial $f\in\C[\mathbf{z}]$, all components of 
$\overline{\mathcal{I}(f)}^{\compl}$ are convex. The number of these convex components is finite.
\end{thm}

\begin{proof}
	Let $C$ be a component of the complement of $\overline{\mathcal{I}(f)}$. Define the holomorphic map
	\begin{equation*}	
	\psi:\C^n\rightarrow\C^n, \quad \mathbf{z}\mapsto \mathbf{z}\cdot e^{-i\frac{\pi}{2}},	
	\end{equation*} 
	which is equivalent to $\textbf{x}+i\textbf{y}\mapsto \textbf{y}-i\textbf{x}$. Furthermore, let
		\begin{equation*}	
		C_\psi \ = \ \psi(\R^n+iC) \ = \ C-i\R^n \ = \ C+i\R^n.	
		\end{equation*}	
	
	We observe that $C_{\psi}$ is a tubular region, that is, for any 	
	$\textbf{y} \in C_{\psi} \cap \R^n$ we have $\textbf{y} + i\textbf{x} \in C_{\psi}$ for 	
	all $\textbf{x} \in \R^n$.	
	Moreover, the function 
	\begin{equation*}
	g:C_{\psi} \to \C , \quad \textbf{w} \mapsto \frac{1}{f(\psi(\textbf{w}))}
	\end{equation*} is holomorphic on $C_\psi$, and $C_\psi$ is the 
	maximal tube with this property. By Bochner's Tube Theorem \cite{bochner}, $g$ is holomorphic on the convex hull 
	of $C_\psi$ (considered as set in $\R^{2n} \cong \C^n$). 
	Due to the maximality of $C_\psi$,
	this implies the convexity of $C_\psi$. Since 	$C_\psi=\psi(\R^n+iC)=C+i\R^n$, we obtain the convexity of $C$.

	As $\mathcal{I}(f)$ is a semialgebraic set, 
the complement of $\mathcal{I}(f)$ is semialgebraic as well.
Then finiteness of the number of convex components follows from
the classical bounds of Oleinik-Petrovski or Milnor-Thom,
see, e.g., \cite[Chapter~7]{bpr-book}. 
\end{proof}	

Theorem~\ref{th:convex} implies the 
following statement on the unbounded components of the complement.
\begin{cor}\label{cor:ray-in-unbounded-complement}
	Every unbounded component of the complement of the closure of an
	imaginary projection contains a ray.
\end{cor}
\begin{proof}
    Let $C$ be an unbounded component of the complement of 
    $\overline{\mathcal{I}(f)}$. Then the convex set $C$ is at least one-dimensional.
    By well-known results in convex analysis (see \cite[Cor.\ II.8.3.1, Thm.\ II.8.4] {rockafellar-book}), 
    the relative interior of $C$ has a recession cone which coincides with the recession cone of the 
    closure of $C$, and that recession cone contains a non-zero vector. Hence, $C$ contains a ray.
\end{proof}
   
The left picture in Figure \ref{fig:boundedandunboundedcomplementsets} shows the imaginary projection of the polynomial $f(z_1,z_2)=z_1^2+z_1^2z_2+2z_1+z_2+1$ with its $6$ convex components
of the complement.
The imaginary projection of a non-constant polynomial is always unbounded, see Section~\ref{sec:infinity}.
As the two right pictures in 
Figure \ref{fig:boundedandunboundedcomplementsets} show, it is possible
that an imaginary projection contains both bounded and unbounded components in
the complement.

\begin{figure}[ht]
	\includegraphics[width=0.25\linewidth]{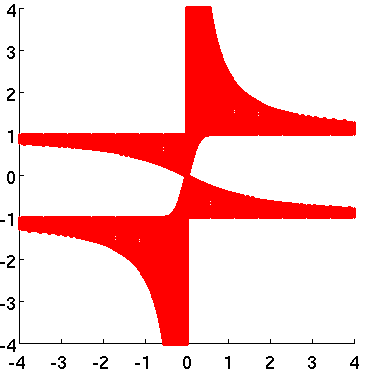}
	\centering
	\includegraphics[height=0.25\linewidth]{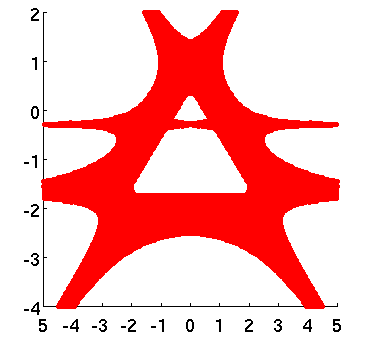} \qquad
\includegraphics[height=0.25\linewidth]{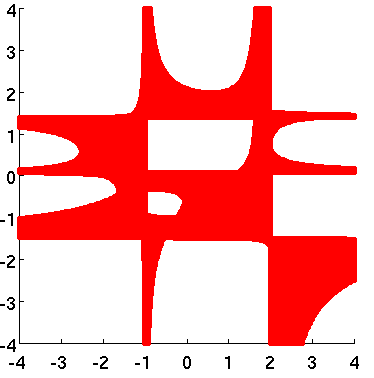}
	\caption{The imaginary projection of
 $f(z_1,z_2)=z_1^2+z_1^2z_2+2z_1+z_2+1$,
$f(z_1,z_2)=z_1^4 + i z_1^3 - z_1^2 z_2^2 + 3 z_1^2 - 2i z_1 z_2^2
+ (4-2i) z_1 + 0.5 z_2^2 + 1.5$
and of
$f(z_1,z_2)=z_1^3z_2^2-iz_1^3z_2+2z_1^3-1.8 z_1^2z_2-1.8iz_1^2+2.1z_1z_2^2+2.1iz_1z_2+4.2z_1-0.2iz_2^2+0.4z_2+1.6i$}
	\label{fig:boundedandunboundedcomplementsets}
\end{figure}

\begin{cor}\label{cor:imag-proj-with-given-number-of-bounded-complement-sets}
	For any integers $n > 0$ and $t>0$, there exists a polynomial 
$f \in \C[\mathbf{z}] = \C[z_1, \ldots, z_n]$ such that the complement of $\mathcal{I}(f)$ has exactly $t$ bounded components.
\end{cor}

\begin{proof}
	We choose an arrangement $\mathcal{H}$ of $d$ hyperplanes $H_1,\ldots,H_d\subseteq\R^n$ such that $\mathcal{H}$ has $t$ bounded components of the complement. 
Using Theorem~\ref{thm:linear-case-imaginary-projection},
each of the hyperplanes is the imaginary projection of an affine-linear polynomial $f_1,\ldots,f_d\in\C[\mathbf{z}]$.
By Lemma \ref{lemma:imaginary-projection-product}, the imaginary projection of the product of these polynomials gives 
exactly the hyperplane arrangement.
\end{proof}

The proof of Corollary \ref{cor:imag-proj-with-given-number-of-bounded-complement-sets} constructed $f$ as a product of linear polynomials. In the following, we investigate the imaginary projection of products of linear polynomials in more detail.

\begin{thm}\label{thm:number-complement-components-product-of-linear-polynomials}
	Let $f \in \C[\mathbf{z}]$ be a product of $m$ affine-linear polynomials in $n$ variables. Then the complement of $\mathcal{I}(f)$ consists of at most $\sum_{k=0}^n\binom{m}{k}$ components, and this bound is tight.
\end{thm}
\begin{proof}By Theorem~\ref{thm:linear-case-imaginary-projection}, the imaginary projection
of an affine-linear polynomial is either a hyperplane or the whole space $\R^n$. We can assume here that the first case holds for every affine-linear polynomial. Then the imaginary projection of
the product defines a hyperplane arrangement in $\R^n$. If the hyperplanes are in general position, 
then they decompose the ambient space into exactly $\sum_{k=0}^n\binom{m}{k}$ many regions, see \cite[Proposition 2.4]{stanley}.
\end{proof}

Theorem~\ref{thm:number-complement-components-product-of-linear-polynomials}
implies the following lower bound for the maximal number of components of the complement of 
$\mathcal{I}(f)$ for polynomials $f$
of total degree $d$ in $n$ variables.

\begin{cor} \label{co:lowerbound}
There exists a polynomial $f$ of total degree $d$ in $n$ variables such that
  the complement of $\mathcal{I}(f)$ consists of exactly $\sum_{k=0}^n \binom{d}{k}$ 
  components.
\end{cor}

For homogeneous polynomials, the components of the complement 
are always unbounded since the imaginary projection is a cone. Furthermore, we have the following relation to hyperbolic polynomials.

	\begin{thm}\label{thm:connection-complement-hyperbolicity-cone}
		Let $f\in\R[\mathbf{z}]$ be homogeneous. Then $\mathcal{I}(f) \neq \R^n$ if and only if $f$ is hyperbolic (with respect to some vector $\mathbf{e} \in \R^n)$.
	\end{thm}
	
	\begin{proof}
          Let $f$ be hyperbolic with respect to $\mathbf{e}\in\R^n$. 
Assuming $\mathbf{e} \in \mathcal{I}(f)$ then implies that
for some $\mathbf{x} \in \R^n$, the imaginary unit $i$ is a root of 
the real function $t \mapsto f(\mathbf{x} + t \mathbf{e})$. This
is a contradiction to the hyperbolicity of $f$.

Conversely, let $\mathbf{e}\not\in \mathcal{I}(f)$. Then we have
$f(\mathbf{x} + i \mathbf{e}) \neq 0$ for all $\mathbf{x} \in \R^n$,
so that in particular
\[
  f(\mathbf{e}) \ = \ (1+i)^{-\deg f} f((1+i) \mathbf{e}) \ = \
  (1+i)^{-\deg f} f(\mathbf{e} + i \mathbf{e}) \neq 0 \, .
\]
Furthermore, if there exists an $\mathbf{x} \in \R^n$ such that
$t \mapsto f(\mathbf{x} + t \mathbf{e})$ has a complex solution
$a + ib$ with $b \neq 0$, then the homogeneous function $f$ would
satisfy
\[
  f(\mathbf{x} + a \mathbf{e} + i b \mathbf{e}) \ = \ 0
\]
in contradiction to $\mathbf{e} \not\in \mathcal{I}(f)$. Hence, $f$ is
hyperbolic with respect to $\mathbf{e}$.
\end{proof}

Similar to Theorem~\ref{thm:connection-complement-hyperbolicity-cone},
for the case of homogeneous polynomials $f \in \R[\mathbf{z}]$,
the components of the complement of $\mathcal{I}(f)$ actually coincide with the
hyperbolicity cones of $f$ (as defined, e.g., in \cite{garding-59}). 
The connection between hyperbolicity cones of homogeneous polynomials
and imaginary connections are explored further in a follow-up
article by the first and the second author \cite{joergens-theobald-hyperbolicity}.

\section{Quadratic and multilinear polynomials\label{se:quadratic}}

In this section, we deal with quadratic and multilinear polynomials.
First, we characterize the imaginary projections of quadratic polynomials with real coefficients. The initial two lemmas reduce the problem
to the imaginary projections of quadratic polynomials in a normal form.

\begin{lemma}\label{le:transf2}
	Let $f\in\C[\mathbf{z}]$ and $A\in\R^{n\times n}$ be an invertible matrix. Then, $\mathcal{I}(f(A\mathbf{z}))=A^{-1} \mathcal{I}(f(\mathbf{z}))$.
\end{lemma}	
\begin{proof}
	Writing $\mathbf{z}=\mathbf{x}+i\mathbf{y}$, the matrix $A$ operates separately on $\mathbf{x}$ and $\mathbf{y}$. Hence,
	\begin{eqnarray*}
	\mathcal{I}(f(A\mathbf{z})) & = & \{\mathbf{y} \, : \, \exists \mathbf{x} \in \R^n \; f(A(\mathbf{x} + i\mathbf{y}))=0\} \ = \
      \{ A^{-1} \mathbf{y'} \, : \, \exists \mathbf{x}' \in \R^n \; f(\mathbf{x'}+i \mathbf{y'})=0\} \\
    & = &
      A^{-1} \mathcal{I}(f(\mathbf{z})).
	\end{eqnarray*}
\vspace*{-5ex}

\end{proof}

\begin{lemma}\label{le:transf1}
	A real translation $\mathbf{z}\mapsto \mathbf{z}+\mathbf{a}$, $\mathbf{a}\in\R^n$, does not change the imaginary projection of a polynomial. An imaginary translation $\mathbf{z}\mapsto \mathbf{z}+i\mathbf{a}$, $\mathbf{a}\in\R^n$, shifts an imaginary projection in direction $-\mathbf{a}$.
\end{lemma}
\begin{proof}
	The statement holds, since the first kind of transformation just translates the real part of the variables and the second one shifts the imaginary parts of the solutions of $f(\mathbf{z})=0$ in direction $-\mathbf{a}$.
\end{proof}

By the Lemmas~\ref{le:transf2} and~\ref{le:transf1}, 
it suffices to study the imaginary projections of polynomials in a normal form in order to
understand the imaginary projections of general quadratic polynomials with real coefficients.
Every real bivariate quadric is affinely 
equivalent  to a quadric given by one of the following polynomials, where the names come 
from the conic sections arising from
considering these polynomials as real polynomials.

\begin{enumerate}
	\item[$(i)$] $z_1^2+z_2^2-1$ (ellipse),
	\item[$(ii)$] $z_1^2-z_2^2-1$ (hyperbola),
	\item[$(iii)$] $z_1^2+z_2$ (parabola),
        \item[$(iv)$] $z_1^2+z_2^2+1$ (empty set),
\end{enumerate}
or one of the special cases
$(v)$ $z_1^2-z_2^2$ (pair of crossing lines),
$(vi)$ $z_1^2-1$ (parallel lines, or a single line $z_1^2$),
$(vii)$ $z_1^2+z_2^2$ (isolated point),
$(viii)$ $z_1^2+1$ (empty set).

In the following theorem, we characterize the imaginary projections of these quadratic polynomials.

\begin{thm}\label{th:quadricr2}
For a quadratic polynomial $f \in \R[z_1,z_2]$, we have
\[
  \mathcal{I}(f) \ = \ \begin{cases}
    \R^2 & \text{if $f$ is of type $(i)$}, \\
    \{-1 \le y_1^2- y_2^2 < 0 \} \cup \{ \mathbf{0} \} & \text{if $f$ is of type $(ii)$}, \\
    \R^2\setminus\{(0,y_2):y_2\neq0\} &
       \text{if $f$ is of type $(iii)$}, \\
    \{y_1^2+y_2^2-1 \ge 0\} & \text{if $f$ is of type $(iv)$}. 
  \end{cases}
\]
In the cases $(v)$ -- $(viii)$, we respectively have
$\mathcal{I}(f)=\{\mathbf{y} \in \R^2 \, : \, y_1^2-y_2^2=0\}$,
$\mathcal{I}(f)=\{\mathbf{y} \in \R^2 \, : \, y_1 = 0 \}$,
$\mathcal{I}(f)=\R^2$,
and $\mathcal{I}(f)=\{\mathbf{y} \in \R^2 \, : \, y_1 = \pm 1\}$.
\end{thm}

\begin{figure}[ht]
	\[
	\includegraphics[height=0.26\linewidth]{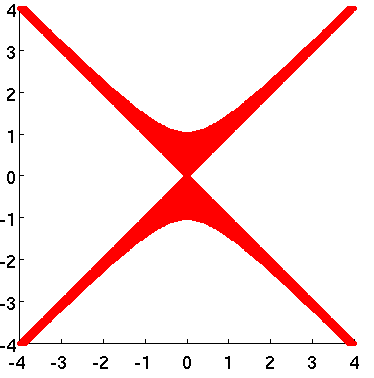} \quad
	\includegraphics[height=0.26\linewidth]{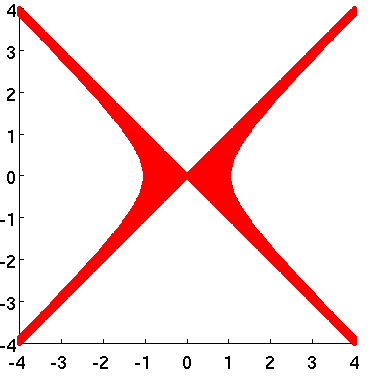}\quad
	\includegraphics[height=0.26\linewidth]{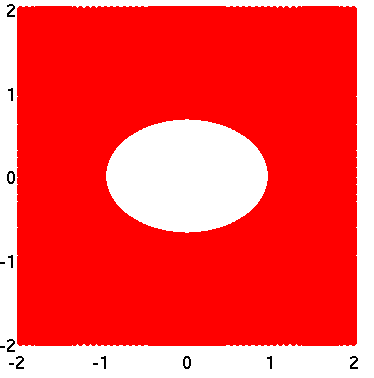}
	\]
	\caption{The imaginary projections of $f(z_1,z_2)=-z_1^2+z_2^2-1$, $f(z_1,z_2)=z_1^2-z_2^2-1$, and
		$f(z_1,z_2)=2z_1^2+z_2^2+1$.}
	\label{fi:ax^2+by^2-1}
\end{figure}

The imaginary projections of some quadratic polynomials are shown in 
Figure \ref{fi:ax^2+by^2-1}, in particular, the middle figure
depicts case~(ii) from Theorem~\ref{th:quadricr2}.

\begin{proof}
For cases $(i)$--$(ii)$ and $(iv)$--$(viii)$, we
consider a polynomial $f = \alpha z_1^2+\beta z_2^2+\gamma$
with $\alpha, \beta, \gamma \in \R$. Decomposing $f(\mathbf{z}) = 0$
into the real and imaginary parts gives 
$\alpha x_1^2 - \alpha y_1^2 + \beta x_2^2 - \beta y_2^2 + \gamma 
= 0$ and
$\alpha x_1 y_1 + \beta x_2 y_2 = 0$. For $y_1 \neq 0$,
eliminating $x_1$ shows that
	\begin{equation}
          \label{eq:quaddesc}
  \mathcal{I}(f) \ = \ 
  \{ \textbf{y} \in \R^n \, : \, 
	x_2^2\left(\alpha \beta y_1^2 + \beta^2y_2^2 \right)=\alpha y_1^2 (\alpha y_1^2+ \beta y_2^2 - \gamma) \text{ has a real solution $x_2$} \} \, .
	\end{equation}
In case~$(i)$, we have $\alpha = \beta = 1$, $\gamma = -1$, which
altogether gives $\mathcal{I}(f) = \R^2$. In case $(ii)$, we have $\alpha = 1$, 
$\beta = \gamma = -1$. For $y_1 \neq 0$,
real solutions for $x_2$ in~\eqref{eq:quaddesc}
exist for $0 < y_2^2 - y_1^2 < 1$ as well as in the special case
$y_1^2 - y_2^2 + 1 = 0$. And for $y_1 = 0$, we obtain
$\mathbf{y} \in \mathcal{I}(f)$ if and only if $y_2^2 \le 1$.

Cases $(iv)$--$(viii)$ can be treated similarly.
Finally, case $(iii)$ is linear in $z_2$, so that the equations for the
real and imaginary part can be solved directly for $x_2$ and $y_2$.
\end{proof}

Now we deal with quadrics in $n$-dimensional space.
Since every quadric in $\R^n$ is affinely equivalent to a quadric given by one of the
following polynomials,
\[
  \begin{array}{cll}
  \text{(I)}& \sum_{j=1}^p z_j^2 - \sum_{j=p+1}^r z_j^2 & \quad (1 \le p \le r, \, r \ge 1, \, p \ge \frac{r}{2}) \, , \\
 \text{(II)} & \sum_{j=1}^p z_j^2 - \sum_{j=p+1}^r z_j^2 + 1 &  \quad (0 \le p \le r, \, r \ge 1) \, , \\
 \text{(III)} & \sum_{j=1}^p z_j^2 - \sum_{j=p+1}^r z_j^2 + z_{r+1} & \quad (1 \le p \le r, \, r \ge 1,
  \, p \ge \frac{r}{2}) \, ,
\end{array}
\]
it suffices to discuss these cases. (See, e.g., \cite{berger-book}
as a general background reference for real quadrics.)

\begin{thm}\label{th:quadrics}
Let $n \ge r \ge 3$ and $f \in \R[\mathbf{z}]$ be a quadratic polynomial.
\begin{enumerate}
\item If $f$ is of type $\mathrm{(I)}$, then
  \begin{equation}\label{eq:class1}
  \mathcal{I}(f) \ = \ \begin{cases}
  \R^n & \text{if } \frac{r}{2} \le p < r-1 \text{ or } p = r \, , \\
  \{\mathbf{y} \in \R^n \; : \, y_r^2\leq  \sum_{j=1}^{r-1} y_j^2\} & \text{if } p = r-1 \, .
  \end{cases}
  \end{equation}
\item If $f$ is of type $\mathrm{(II)}$, then
  \begin{equation}\label{eq:class2}
  \mathcal{I}(f) \ = \ \begin{cases}
  \R^n & \text{if } p = 0 \text{ or } 1 < p < r-1 \, , \\
  \{\mathbf{y} \in \R^n \; : \, y_1^2 - \sum_{j=2}^{r} y_j^2 \le 1\} & \text{if } p = 1 \, , \\
   \{\mathbf{y} \in \R^n \; : \, \sum_{j=1}^{r-1} y_j^2 > y_r^2 \} \cup \{\mathbf{0}\} & \text{if } p = r-1 \, , \\
   \{\mathbf{y} \in \R^n \; : \, \sum_{j=1}^{r} y_j^2 \ge 1\} & \text{if } p = r \, .
  \end{cases}
\end{equation}
\item If $n > r$ and $f$ is of type $\mathrm{(III)}$,  then 
$$\mathcal{I}(f) \ = \ \R^n \setminus \{ (0,\ldots,0,y_{r+1},y_{r+2}, \ldots, y_n) \, : \, y_{r+1} \neq 0, \, y_{r+2}, \ldots, y_n \in \R \}.$$
\end{enumerate}
\end{thm}

Note that the case $n \ge 3$ differs significantly from $n=2$.
The proof of Theorem~\ref{th:quadrics} is given in the
Lemmas~\ref{thm:quadric-cone}--\ref{le:hyperbolicparaboloid}.

\begin{lemma}\label{thm:quadric-cone}
	Let $n\geq r \ge 3$. If $f(\mathbf{z})=\sum_{j=1}^pz_j^2-\sum_{j=p+1}^{r}z_j^2$ 
      with $\frac{r}{2} < p\leq r$, then $\mathcal{I}(f)$ is given 
  by~\eqref{eq:class1}.
\end{lemma}

\begin{proof}
Without loss of generality we can assume $r=n$.
	Splitting the problem into the real and imaginary part yields
	\begin{eqnarray}
	\sum_{j=1}^{p} x_j^2- \sum_{j=p+1}^n x_j^2 - 
       \sum_{j=1}^{p} y_j^2 + \sum_{j=p+1}^{n} y_j^2 & = & 0 \, , \label{eq:hyperboloid-1}\\
	\sum_{j=1}^{p} x_jy_j - \sum_{j=p+1}^{n} x_jy_j  & = & 0 \label{eq:hyperboloid-2}\, .
	\end{eqnarray}
  Consider a fixed $\textbf{y} \in\R^n$. 
  If $- \sum_{j=1}^{p} y_j^2 +  \sum_{j=p+1}^n y_j^2=0$, then 
  $x_1 = \cdots = x_n = 0$ gives a solution to~\eqref{eq:hyperboloid-1} and~\eqref{eq:hyperboloid-2}. Therefore, we can assume $- \sum_{j=1}^{p} y_j^2 + \sum_{j=p+1}^n y_j^2\neq0$. 

In the case $p = n$, by reordering the indices, we can assume that $y_1^2 + y_2^2 \neq 0$,
and choose $\textbf{x} = \frac{\|\textbf{y}\|_2}{(y_1^2+y_2^2)^{1/2}}(-y_2,y_1,0, \ldots, 0)$
to obtain a solution for~\eqref{eq:hyperboloid-1} and~\eqref{eq:hyperboloid-2}.

In the case $p = n-1$, the $n$-dimensional hyperboloid~\eqref{eq:hyperboloid-1} 
in the $x$-variables is one-sheeted for $- \sum_{j=1}^{n-1} y_j^2 +  y_n^2<0$ and two-sheeted for $- \sum_{j=1}^{n-1} y_j^2 +  y_n^2>0$. 
In case of a one-sheeted hyperboloid, its intersection with the hyperplane~\eqref{eq:hyperboloid-2} is never empty. Namely, choosing $x_3=\cdots=x_{n-1}=0$, gives the hyperboloid $x_1^2+x_2^2-x_n^2 = \sum_{j=1}^{n-1} y_j^2 - y_n$ in $x_1,x_2,x_n$, that contains the origin in the inner component of its complement.

	Now consider the case where the hyperboloid consists of two sheets. For any $\alpha > 0$, 
the sets $\{\mathbf{y} \in \R^n \, : \, - \sum_{j=1}^{n-1} y_j^2 + y_n^2=\alpha>0\}$ and $\{\mathbf{x} \in \R^n \, : \, \sum_{j=1}^{n-1} x_j^2 - x_n^2=-\alpha\}$ coincide. Furthermore, after a coordinate transformation we can assume $\alpha = 1$ and set $H = \{\mathbf{x} \in \R^n \, : \, -\sum_{j=1}^{n-1} x_j^2 + x_n^2=1\}$. 

We claim that the intersection of $H$ with the hyperplane~\eqref{eq:hyperboloid-2} is always empty.
Due to the symmetry of $H$ with respect to all the coordinate hyperplanes $x_k=0$ for $1 \le k \le n-1$,
it suffices by~\eqref{eq:hyperboloid-2}
to show that the hyperboloid $H$ does not contain two distinct points, whose
position vectors are orthogonal to each other with respect to the Euclidean scalar product.
Because of the rotational symmetry of $H$ with regard to the $x_n$-axis and the invariance
of scalar products under orthogonal transformations,
by applying an orthogonal
transformation it suffices to consider the situation $x_2 = \cdots = x_{n-1} = 0$. 
The resulting hyperbola $-x_1^2+x_n^2=1$ in the $x_1$-$x_n$-plane has no two orthogonal 
position vectors. Namely, the asymptotes $x_1=\pm x_n$ divide the plane into four quarters,
and the hyperbola is contained in the strict interiors of two opposite 
quarters.

Now consider the case $\frac{n}{2} < p < n-1$. By our initial considerations in the proof, we have already covered the case $\mathbf{y} = \mathbf{0}$. 
In the case
$\mathbf{y} \neq \mathbf{0}$, by changing the coordinates we can assume that
$(y_1,y_2,y_{p+1})$ is not the zero vector.
Choose $(x_{p+2},\ldots,x_n)\in\R^{n-p-1}$ such that $-\sum_{j=p+2}^n x_j^2 - \sum_{j=1}^p y_j^2 + \sum_{j=p+1}^n  y_j^2=:\alpha<0$. Then, since $\mathbf{y}$ is fixed,
\eqref{eq:hyperboloid-1} becomes a hyperboloid of one sheet and \eqref{eq:hyperboloid-2} becomes an affine hyperplane. The intersection of these two hypersurfaces is non-empty. Namely, choosing $x_3=\cdots = x_{p}=0$, gives the one-sheeted hyperboloid $\{(x_1,x_2,x_{p+1}) \in \R^3 \, : \, x_1^2+x_2^2-x_{p+1}^2=-\alpha>0\}$, which intersects the affine hyperplane with normal vector $(y_1,y_2,-y_{p+1})$ and constant term $- \sum_{j=p+2}^n  x_jy_j$.
Hence, there exists an $\mathbf{x} \in \R^n$ satisfying~\eqref{eq:hyperboloid-1} 
and~\eqref{eq:hyperboloid-2}.
\end{proof}

\begin{lemma}
	Let $n\geq r \ge 3$ and $f(\mathbf{z})=\sum_{j=1}^p z_j^2 - \sum_{j=p+1}^r z_j^2 + 1$ with 
  $0 \le p \le r$. 
	\begin{enumerate}
       \item If $p = 0$ then $\mathcal{I}(f) = \R^n$.
       \item If $p = 1$ then $\mathcal{I}(f) = \{\mathbf{y} \in \R^n \, : \, y_1^2 - \sum_{j=2}^r y_j^2 \le 1 \}$. 
	\item If $1<p<r-1$ then $\mathcal{I}(f)=\R^n$.
	\item If $p=r-1$ then $\mathcal{I}(f)=
	\{\mathbf{y} \in \R^n \; : \, \sum_{j=1}^{r-1} y_j^2 > y_r^2 \} \cup \{\mathbf{0}\}$.
	\item If $p = r$ then $\mathcal{I}(f) = \{\mathbf{y} \in \R^n \, : \,  \sum_{j=1}^r y_j^2 \ge 1\}$.
	\end{enumerate}
\end{lemma}

Note that for $r = 3$ the case (3) cannot occur.

\begin{proof}
Similar to the proof of Lemma~\ref{thm:quadric-cone}, we can assume $r=n$
and split the problem into the real and imaginary part,
	\begin{eqnarray}
	\sum_{j=1}^{p} x_j^2- \sum_{j=p+1}^n x_j^2 - 
       \sum_{j=1}^{p} y_j^2 + \sum_{j=p+1}^{n} y_j^2 +1 & = & 0 \, , \label{eq:hyperboloid-2-1}\\
	\sum_{j=1}^{p} x_jy_j - \sum_{j=p+1}^{n} x_ny_n  & = & 0 \label{eq:hyperboloid-2-2}\, .
	\end{eqnarray}
Consider a fixed $\textbf{y} \in\R^n$. 

In the case $p=0$, we obtain the
 two equations $\sum_{j=1}^n x_j^2 = \sum_{j=1}^n y_j^2+1$ and $\sum_{j=1}^n x_jy_j=0$. 
Setting
$\textbf{x} = \big( \frac{\|\textbf{y}\|_2^2+1}{y_1^2+y_2^2} \big)^{1/2}(-y_2,y_1,0, \ldots, 0)$
gives a solution.

In the case $p=1$, set $\alpha =-y_1^2+\sum_{j=2}^{n} y_j^2+1$. Then the statement follows identically as in Lemma~\ref{thm:quadric-cone} in the cases $\alpha =0$, $\alpha<0$ and $\alpha>0$. For $\alpha=0$, the point $\mathbf{x}=\mathbf{0}$ is a solution for $f(\mathbf{z})=f(\mathbf{x}+i\mathbf{y})=0$. For $\alpha>0$, \eqref{eq:hyperboloid-2-1} is a one-sheeted hyperboloid and \eqref{eq:hyperboloid-2-2} is a hyperplane. Their intersection is non-empty. For $\alpha<0$, the formula for $\alpha$ and
\eqref{eq:hyperboloid-2-1} both define two-sheeted hyperboloids. We consider the hyperboloids $H_1:=\{\mathbf{y} \in \R^n \, : \, y_1^2-\sum_{j=2}^{n} y_j^2=1-\alpha\}$ and $H_2:=\{\mathbf{x} \in \R^n \, : \, x_1^2-\sum_{j=2}^n x_j^2=-\alpha\}$. Via the transformations $\mathbf{y}\mapsto \mathbf{y} / \sqrt{1-\alpha}$ and $\mathbf{x}\mapsto \mathbf{x} / \sqrt{-\alpha}$ these sets are transformed into the same set $C = \{\mathbf{x} \in \R^n \, : \, x_1^2-\sum_{j=2}^n x_j^2=1\}$. We know by the proof of Lemma \ref{thm:quadric-cone} that there is no pair of orthogonal position vectors
on $C$. Therefore, there are no orthogonal position vectors 
in $H_1$ and $H_2$. Hence, for $\alpha<0$ the equation $f(\mathbf{z})=0$ has no solution in $\mathbf{x}$. 
The case $p=r-1$ is similar.

In the case $1 < p < r-1$, the statement follows as in Lemma~\ref{thm:quadric-cone}.

In the case $p=n$, there exists an~$\textbf{x}$ satisfying ~\eqref{eq:hyperboloid-2-1} 
and~\eqref{eq:hyperboloid-2-2} if and only if $\sum_{j=1}^n y_j^2 - 1 \ge 0$.
\end{proof}

\begin{lemma}\label{le:hyperbolicparaboloid}
	Let $n > r \geq2$. If $f(\mathbf{z})=\sum_{j=1}^pz_j^2-\sum_{j=p+1}^{r}z_j^2+z_{r+1}$ 
      with $1\leq p\leq r$, then 
      $\mathcal{I}(f)=\R^n \setminus \{ (0,\ldots,0,y_{r+1},y_{r+2}, \ldots, y_n) \, : \, y_{r+1} \neq 0, \, y_{r+2}, \ldots, y_n \in \R \}$.
\end{lemma}

\begin{proof}We can assume $n=r+1$.
  In the system for the real and the imaginary parts
	\begin{eqnarray}
	\sum_{j=1}^p x_j^2-\sum_{j=p+1}^{n-1} x_j^2-\sum_{j=1}^p y_j^2 + \sum_{j=p+1}^{n-1} y_j^2+x_n=0 \ , \label{eq:hyperbolic-parabolic-3-1}\\
	2 \sum_{j=1}^p x_jy_j - 2 \sum_{j=p+1}^{n-1} x_jy_j+y_n=0 \ , \label{eq:hyperbolic-parabolic-3-2}\end{eqnarray}
consider a fixed $(y_1,\ldots,y_n)\in\R^n$. 
If $(y_1, \ldots, y_{n-1}) \neq \mathbf{0}$, then we can choose
$(x_1,\ldots,x_{n-1})\in\R^n$ such that \eqref{eq:hyperbolic-parabolic-3-2} is satisfied. 
Since~\eqref{eq:hyperbolic-parabolic-3-1} is linear in $x_n$, it has a real solution for $x_n$.
In the special case $(y_1, \ldots, y_{n-1}) = \mathbf{0}$, we see that
$\mathbf{y} = (y_1, \ldots, y_n) \in \mathcal{I}(f)$ if and only if $y_n = 0$.  
\end{proof}

Lemmas~\ref{le:transf2} and~\ref{le:transf1} 
also provide a statement about the existence of unbounded components in 
the complement.

\begin{thm} Let $f \in \C[\mathbf{z}]$.
\begin{enumerate}
\item The complement of $\mathcal{I}(f)$ contains the non-negative $y_1$-axis 
$\R_{\ge 0} \times \{0\}^{n-1}$ 
if and only if the polynomial $f(z_1 + ir, z_2, \ldots, z_n)$ 
has no real solution in $\mathbf{z}$ for any $r \ge 0$.
\item $\overline{\mathcal{I}(f)}$ has an unbounded component in the complement if and only if there is an
  affine transformation $\mathbf{z} \mapsto A \mathbf{z} + i\mathbf{b}$ with a 
  real matrix $A$ and a real vector $\mathbf{b}$ such that condition $(1)$ is satisfied.
\end{enumerate}
\end{thm}
\begin{proof}
The first statement immediately follows from the definition of 
$\mathcal{I}(f)$. For the second statement, 
Corollary~\ref{cor:ray-in-unbounded-complement} implies that
the existence of an unbounded component in the complement is equivalent to
the existence of a ray in the complement.
By the Lemmas~\ref{le:transf2} and~\ref{le:transf1}, the affine
transformation reduces the situation to (1).
\end{proof}

\subsection*{Multilinear polynomials}

We study the imaginary projection of multilinear polynomials
(in the sense of multi-affine-linear). 
Br\"and\'{e}n's stability result for this class was given in 
Theorem~\ref{th:braenden1}. The next statement
describes the imaginary projection of bivariate multilinear polynomials; 
see the right picture in Figure~\ref{fi:top1} for an example. 

\begin{thm}\label{thm_hyperbola}
        Let $f(z_1,z_2) = z_1z_2 + \beta z_1 + \gamma z_2 + \delta$ be a multilinear
        polynomial with $\beta, \gamma, \delta \in \R$.
        Then
        \[
          \mathcal{I}(f) \ = \ \Big\{\mathbf{y} \in \R^2 \, : \, 0 < \frac{y_1 y_2}{\delta - \beta \gamma} \le 1 \Big\} \cup \{ \mathbf{0} \} 
           \qquad \text{for } \delta - \beta \gamma \, \neq \, 0 \, .
        \]
        In the special case $\delta = \beta \gamma$,
        the multilinear polynomial is reducible
        and thus
        $\mathcal{I}(f) = \mathcal{I}(z_1 + \gamma) \cup \mathcal{I}(z_2 + \beta)
        = \mathcal{I}(z_1) \cup \mathcal{I}(z_2) =
        (\R \times \{0\}) \cup (\{0\} \times \R)$.
\end{thm}

As a consequence, we rediscover that the multilinear polynomial $f$ is stable 
if and only if
$\beta \gamma - \delta \ge 0$, see Theorem~\ref{th:braenden1}.

\begin{proof}
Since $f$ can be written as $f(z_1,z_2) = (z_1 + \gamma)(z_2 + \beta) + \delta - \beta \gamma$,
Lemma~\ref{le:transf1} implies that $\mathcal{I}(f) = \mathcal{I}(g)$ where
$g(z_1,z_2) = z_1 z_2 + \delta - \beta \gamma$.
Substituting $z_1 = z_1' + z_2'$ and $z_2 = (\beta \gamma - \delta)(z_1' - z_2')$,
we can express $g$ as $g(z_1',z_2') = (\beta \gamma - \delta)(z_1^2 - z_2^2 -1)$, and
by~Theorem~\ref{th:quadricr2},
the imaginary projection of $g$ with respect to the $\mathbf{z'}$-variables is
\[
  \{\mathbf{y'} \in \R^2 \, : \, -1 \le (y'_1)^2 - (y'_2)^2 < 0 \} \cup \{ \mathbf{0} \} \, .
\]
Using $\begin{pmatrix} 1 & 1 \\ \beta \gamma - \delta & -(\beta \gamma - \delta)
\end{pmatrix}^{-1} 
= \frac{1}{2} \begin{pmatrix} 1 & 1/(\beta \gamma - \delta) \\ 1 & -1/(\beta \gamma - \delta) \end{pmatrix}$,
transforming back to the $\mathbf{z}$-va\-ri\-ables with Lemma~\ref{le:transf2} yields the claim.
\end{proof}

For the case of $n$-dimensional multilinear polynomials, we provide the
subsequent, less explicit, characterization of the imaginary projection, and more generally,
of polynomials of the form $f=g+z_{n+1}h\in\C[\mathbf{z},z_{n+1}]$ with 
$g, h \in \C[\mathbf{z}]$.

\begin{lemma}
Let $f=g+z_{n+1}h\in\C[\mathbf{z},z_{n+1}]$ and $v \in \R$. 
A point $(\mathbf{z},z_{n+1})$ with $\Im z_{n+1} = v$ and $h(\mathbf{z}) \neq 0$
is contained in $\mathcal{V}(f)$ if and only if the determinant
\begin{equation}
\label{det-condition}
	\det\left(\begin{matrix}
	\Re g- v \Im h & \Re h\\ \Im g +  v \Re h & \Im h
	\end{matrix}\right)
\end{equation}
vanishes in $\mathbf{z}$.
\end{lemma}

\begin{proof}
Writing $z_{n+1} = u+iv$, the conditions $\Re f=0$ and $\Im f=0$ give
	\[
	\left(\begin{matrix}\Re g\\ \Im g\end{matrix}\right)+\left(\begin{matrix}\Re h & -\Im h\\ \Im h & \Re h\end{matrix}\right)\left(\begin{matrix}u\\ v\end{matrix}\right)=\left(\begin{matrix}\Re g-v\Im h\\ \Im g+v\Re h\end{matrix}\right)+u\left(\begin{matrix}\Re h\\ \Im h\end{matrix}\right) \ = \ \mathbf{0} \, .
	\]
Considering this equation as a linear equation in $u$ shows that there
exists a solution if and only if the coefficient vector and the constant vector
are linearly dependent, that is, if and only 
the determinant~\eqref{det-condition} vanishes.
\end{proof}

We obtain the following corollary.

\begin{cor}\label{cor:imag-proj-of-multiaffine-poly}
Let $f=g+z_{n+1}h\in\C[\mathbf{z},z_{n+1}]$. 
Then, writing $\mathbf{z} = \mathbf{x} + i \mathbf{y}$,
the sets $\mathcal{I}(g+z_{n+1}h)$ and
\begin{equation} \label{det-condition2}
  \left\{(\mathbf{y},v)\in\R^{n+1} \, : \, 
\exists \mathbf{x} \in \R^n \text{ with } \, \det\left(\begin{matrix}\Re g-v\Im h & \Re h\\ \Im g+v\Re h & \Im h\end{matrix}\right)=0\right\}
\end{equation}
coincide outside of the exceptional set $E = \{ \Im((\mathbf{z},z_{n+1})) \, : \, h(\mathbf{z}) = 0 \text{ and } g(\mathbf{z}) \neq 0 \}$.
\end{cor}

We observe that the determinantal condition in~\eqref{det-condition}
and~\eqref{det-condition2} gives a linear condition in $v$. 
For a multilinear polynomial of the form 
$f=g+z_{n+1}h\in\R[\mathbf{z}, z_{n+1}]$ 
with $g$ and $h$ multilinear, the condition is quadratic in
any of the variables $\mathbf{x} = (x_1, \ldots, x_n)$ and
 $\mathbf{y} = (y_1, \ldots, y_n)$.

\begin{example}
We revisit the multilinear polynomial $f(z_1,z_2)=z_1z_2+\delta$, $\delta\in\R \setminus \{0\}$ to illustrate Corollary \ref{cor:imag-proj-of-multiaffine-poly}; see Theorem~\ref{thm_hyperbola}. 
Setting $g=\delta$ and $h=z_1$, the determinantal condition~\eqref{det-condition} 
gives (where we write $y_2$ instead of $v$)
\[
	\delta y_1-y_1^2y_2-x_1^2y_2 \ = \ 0 \, .
\]
For $y_2 \neq 0$, there exists a real solution for $x_1$ if and only if
$\frac{y_1}{y_2}(\delta-y_1y_2)\geq0$. Taking into account the
exceptional set $E = \{0\} \times \R$, we obtain
$\mathcal{I}(f) = \{ \mathbf{y} \in \R^2 \, : \,
0 < \frac{y_1 y_2}{\delta} \le 1\} \cup \{ \mathbf{0} \}$,
in accordance with Theorem~\ref{thm_hyperbola}.
\end{example}

\begin{example}
We consider the non-multilinear polynomial
$f(z_1,z_2) = 1+z_2 z_1^2$, which is of the form $f = g + z_2 h$ with
$g=1$ and $h=z_1^2$. Corollary \ref{cor:imag-proj-of-multiaffine-poly}
gives the quartic condition in the variable $x_1$
\begin{equation}
  \label{eq:quarticcond}
  - y_2 x_1^4 - 2 y_1^2 y_2 x_1^2 + 2 y_1 x_1 - y_2 y_1^4 \ = \ 0 \, .
\end{equation}
Recall that the discriminant of a general polynomial
$p(z) = \sum_{j=0}^n a_j z^j$ is given
by $\Disc(p) = (-1)^{\frac{1}{2}n(n-1)}\frac{1}{a_n} \Res(p,p')$,
where $\Res$ denotes the resultant. For a quartic, a positive discriminant
corresponds to zero or four real roots, while a negative discriminant
corresponds to two real roots. Moreover, with the notation
\[
  H \ = \ 8 a_2 a_4 - 3 a_3^2 \, , \quad
  I \ = \ 12 a_0 a_4 - 3 a_1 a_3 + a_2^2 \, ,
\]
the case of four real roots corresponds to $H \le 0$ and $H^2 - 16 a_4^2 I \ge 0$,
while the case of four complex roots corresponds to  
$H > 0$ or $H^2 - 16 a_4^2 I < 0$, see, e.g., \cite[Proposition 7]{cremona-99}.
In our situation, $p = p(x_1)$ is the polynomial in~\eqref{eq:quarticcond}, 
$\Disc(p) = 16 (64 y_1^4 y_2^2 - 27) y_1^4 y_2^2$,
$H = 16 y_1^2 y_2^2$ and $H^2 - 16 a_4^2 I = 0$. The set of points $\mathbf{y} \in \R^2$,
where~\eqref{eq:quarticcond} has at least two real solutions in $x_1$,
is given by $64 y_1^4 y_2^2 \le 27$. Taking into account the exceptional
set $E = \{0\} \times \R$ gives
\[
  \mathcal{I}(f) \ = \ \{ \mathbf{y} \in \R^2 \, : \,
    0 < 64 y_1^4 y_2^2 \le 27 \} \cup (\R \times \{0\}) \, .
\]
\end{example}

We will return to multilinear polynomials when studying their asymptotic geometry
in Theorem~\ref{thm:limits-deg-n}.

\section{The limit set of imaginary projections}\label{sec:infinity}

For the amoeba $\mathcal{A}(f)$ of a polynomial $f$ it is well-known that the set
of limit points of points in $\frac{1}{r} \mathcal{A}(f) \cap \Sph^{n-1}$, 
where $r > 0$ tends to infinity,
is a spherical polyhedral complex. It is called the \emph{logarithmic limit set}
\[
  \mathcal{A}_\infty(f) \ = \ 
  \lim_{r\rightarrow\infty}\left(\frac{1}{r}\mathcal{A}(f)\cap\Sph^{n-1}\right)
\]
and provides
one way of defining a tropical hypersurface; see, e.g., \cite[Section 1.4]{maclagan-sturmfels-book}.

For imaginary projections, the situation is different from amoebas,
as shown by the following counterexample: For $f \in \C[\mathbf{z}]$,
	\[\mathcal{I}_\infty(f) \ = \ \lim_{r\rightarrow\infty}\left(\frac{1}{r}\mathcal{I}(f)\cap \Sph^{n-1}\right)\]
	is not a spherical polyhedral complex in general.

\begin{example}
	Let $f(\mathbf{z})=z_1^2 - \sum_{j=2}^n z_j^2 +1$ with 
$n \ge 3$. Then, by Theorem~\ref{th:quadrics}, 
	\[
	\mathcal{I}(f) \ = \{ \mathbf{y} \in \R^n \, : \, y_1^2 - \sum_{j=2}^n y_j^2 \le 1 \} \, .
\]
	Therefore, $\mathcal{I}_\infty(f) = 
      \lim_{r \to \infty} \big \{ \mathbf{y} \in \Sph^{n-1} \, : \,
          (ry_1)^2 - \sum_{j=2}^n (ry_j)^2 \le 1 \big\}$
can be written as
\[
  \mathcal{I}_{\infty}(f) \ = \
   \Big\{ \mathbf{y} \in \Sph^{n-1} \, : \, y_1^2 \le \sum_{j=2}^n y_j^2 \Big\} 
  \ = \ \Big\{ \mathbf{y} \in \Sph^{n-1} \, : \, y_1^2 \le \frac{1}{2} \Big\} \, .
\]
    Since $n \ge 3$, this cannot be written as the intersection of $\Sph^{n-1}$ with a polyhedral
    fan. Hence, $\mathcal{I}_{\infty}(f)$ is not a spherical polyhedral complex, and since 
    $\mathcal{I}_{\infty}(f)$ is already closed, this persists under taking the closure.
\end{example}

\begin{definition}
	Let $f\in\C[\mathbf{z}]$. We call a point $p\in\R^n$ 
      a \emph{limit direction} of the imaginary projection of $f$ if $p\in\mathcal{I}_\infty(f)$.
\end{definition}

Unless $f$ is univariate or constant, $f$ has at least one limit
direction. 
Namely, for 
any integer $N > 0$ and 
$(z_1, \ldots, z_{n-1}) \in \C^{n-1}$ such that
$\|(\Im(z_1), \ldots, \Im(z_{n-1}))\|_2 > N$ and
$f(z_1, \ldots, z_{n-1},z_n) \in \C[z_n]$ is not a non-zero constant,
there exists a $z_n \in \C$ with $f(z_1, \ldots, z_{n-1}, z_n) = 0$.
The resulting sequence of points $\mathbf{z}_N$ induces a sequence
of points
$\Im(\mathbf{z}_N)/\|\Im(\mathbf{z}_N)\|$ on the unit sphere
$\Sph^{n-1}$. By compactness, there exists a convergent subsequence.
If $f$ has real coefficients, then, by Remark~\ref{re:symm}, the limit
directions are symmetric with respect to the origin.

In Theorem~\ref{thm:limits-deg-n} and Corollary~\ref{co:multilin-unbounded}, we deal 
with the limit directions of multilinear polynomials. Then, in 
Theorem~\ref{thm:limit-directions-general-case}
and Corollary~\ref{co:bivar-limit}, we provide criteria for one-dimensional families
of limit directions, which means in the case $n=2$ that every point on $\Sph^1$ is a limit
direction of the imaginary projection.

\begin{thm}\label{thm:limits-deg-n}
	Let $f\in\C[\mathbf{z}]$ be a multilinear polynomial, and assume that the 
monomial $z_1\cdots z_n$ appears in $f$, i.e., $\deg(f)=n$. Then the limit directions of
$\mathcal{I}(f)$ are given by $\Sph^{n-1} \cap \mathcal{H}$, where $\mathcal{H}$
is the union of the $n$ coordinate hyperplanes $\{ \mathbf{y} \in \R^n \, : \, y_j = 0\}$, $1 \le j \le n$.
\end{thm}

\begin{proof}
Homogenizing $f$ to $f_h(z_0,z_1,\ldots,z_n)$, the homogeneous polynomial 
$f_h$ has a zero at infinity, i.e., $(0,z_1, \ldots, z_n) \in \mathcal{V}(f_h)$,
if and only if $z_1\cdots z_n=0$. Hence, the set of limit points of points in 
$\frac{1}{r} \mathcal{V}(f)$, $r \to \infty$, is
$\mathcal{V}(z_1 \cdots z_n)$. The imaginary projections of the $n$ hyperplanes 
$\{ \mathbf{z} \in \C^n \, : \, z_j = 0\}$ then imply the claim. 
\end{proof}

Theorem ~\ref{thm:limits-deg-n} allows one to characterize the number of unbounded 
components in the complement of the imaginary projection of multilinear polynomials.

\begin{cor}\label{co:multilin-unbounded}
	Let $f\in\C[\mathbf{z}]$ be a multilinear polynomial, and assume that the monomial 
$z_1\cdots z_n$ appears in $f$, i.e., $\deg(f)=n$. Then the complement of 
$\mathcal{I}(f)$ contains exactly $2^n$ unbounded components.
\end{cor}
We remark that this number coincides with the number stated in Theorem \ref{thm:number-complement-components-product-of-linear-polynomials}  when choosing $m=n$.
\begin{proof}
	By Theorem \ref{thm:limits-deg-n},
the complement of
$\mathcal{I}_{\infty}(f)$ consists of $2^n$ components. 
Therefore, the complement of $\mathcal{I}(f)$ has exactly $2^n$ unbounded components.
\end{proof}

\begin{thm}\label{thm:limit-directions-general-case}
        Let $f\in\C[\mathbf{z}]$ be a non-constant polynomial. 
If its homogenization $f_h\in \C[z_0,\mathbf{z}]$ has a zero
$\mathbf{p}_h=(0:p)=(0:p_1:\cdots:p_n)\in\P_\C^{n}$, 
then every point in the intersection $\Sph^{n-1}\cap\mathcal{H}$ is a limit direction, where
$\mathcal{H}=\{\lambda \Re(\mathbf{p})+\mu \Im(\mathbf{p}):\lambda,\mu\in\R\}$
and $\mathbf{p} = (p_1, \ldots, p_n)$.
\end{thm}
\begin{proof}
	Let $\mathbf{p}_h = (0:p_1:\cdots:p_n)$ be a zero at infinity of $f_h$.
Since $i \mathbf{p}_h$ is also a point at infinity for $f_h$, we can assume
that $\sum_{j=1}^n \Im(p_j)^2 \neq 0$. Since $f$ is non-constant,
there exists a sequence $(\mathbf{p}^{(k)}) = (p_1^{(k)},\ldots, p_n^{(k)})$ of points in $\mathcal{V}(f)$ such that $(1:p_1^{(k)}:\cdots:p_n^{(k)})$ converges to $\mathbf{p}_h$. Hence,
	\[
	\frac{1}{(\sum_{j=1}^n\Im(p_j)^2)^{1/2}}(\Im p_1,\ldots,\Im p_n)
	\] 
	is a limit direction. 
	Multiplying $\mathbf{p}_h$ with a complex number $\mu+i\lambda$, $\lambda,\mu \in \R$ keeps $\mathbf{p}_h$ invariant, and under the imaginary projection it leads to a projected point 
$\Im ( (\mu+i\lambda)\cdot \mathbf{p}) = \mu \Im(\mathbf{p}) + \lambda \Re(\mathbf{p})$.
Considering all complex numbers $\mu + i \lambda \in\C$, these points form the 
subspace $\mathcal{H}$. 
\end{proof}

\begin{example}
	We revisit the polynomial $f(z_1,z_2)=z_1^2-z_2^2-1$; see Figure \ref{fi:ax^2+by^2-1} for its imaginary projection. 
Its homogenization is $f_h(z_0,z_1,z_2)=z_1^2-z_2^2-z_3^2$ whose zeros at infinity are given by the equation $z_1^2-z_2^2=(z_1+z_2)(z_1-z_2)=0$.
For the points $\mathbf{p}_h=(0:1:\pm 1)$,
Theorem \ref{thm:limit-directions-general-case} provides the two one-dimensional lines $\mathcal{H}_{1,2}=\{\lambda(1,\pm1):\lambda\in\R\}$. 
We obtain the intersection $\Sph^1\cap\mathcal{H}_{1,2}=\{(\pm\frac{1}{\sqrt{2}},\pm\frac{1}{\sqrt{2}})\}$.
Indeed, we know by Theorem \ref{th:quadricr2} that $\mathcal{I}(f)=\{\mathbf{y} \in \R^2 \, : \,
y_1^2-y_2^2=0\}$, which
confirms $\mathcal{I}_\infty(f)=\{(\pm\frac{1}{\sqrt{2}},\pm\frac{1}{\sqrt{2}})\}$. 
\end{example}

Theorem \ref{thm:limit-directions-general-case} implies the following statement about bivariate polynomials of arbitrary degree:

\begin{cor} \label{co:bivar-limit}
	Let $f\in\C[z_1,z_2]$ be of total degree $d$ and assume its homogenization $f_h$ has the zeros at infinity $(0:1:a_j)$, $j=1,\ldots,d$. Then,
\[
  \mathcal{I}_\infty(f) \ = \ \begin{cases}
     \bigcup\limits_{j=1}^d \left\{\pm\frac{1}{\sqrt{1+a_j^2}}(1,a_j)\right\} & 
        \text{if all $a_j$ are real,}\\
          \Sph^1 & \text{otherwise.}
   \end{cases}
\]
\end{cor}

Note that by changing coordinates, zeros of the form $(0:0:b_j)$ with
$b_j \neq 0$ are covered by the statement as well.

\begin{proof}
	Assume first that there is an $a_j$ with $\Im(a_j) \neq 0$, say $a_1$. Then, by Theorem \ref{thm:limit-directions-general-case}, the subspace
 $\mathcal{H} = \{ \lambda (1,\Re(a_1)) + \mu(0,\Im(a_1)) \, : \,
  \lambda, \mu \in \R\}$ is two-dimensional and thus the set of limit directions is
  $\mathcal{H}\cap\Sph^1=\Sph^1$.

	If all $a_j$ are real, then all the subspaces $\mathcal{H}_j$ corresponding to the points 
$(0:1:a_j)$ are one-dimensional. The intersection $\mathcal{H}_j\cap\Sph^1$ contains the
points $\pm\frac{1}{\sqrt{1+a_j^2}}(1,a_j)$. 

In order to show that there are no further limit directions,
let $(\mathbf{p}^{(n)})_{n \in \N}$ be a sequence of points in $\mathcal{V}(f)$ with
$\|\Im(\mathbf{p}^{(n)})\|_2 \to \infty$. Since the curve $\mathcal{V}(f)$ has only a finite
number of points in the plane at infinity, namely $d$, the sequence $(\mathbf{p}^{(n)})_{n \in \N}$
can be decomposed
into $d$ disjoint subsequences $(\mathbf{q}_1^{(n)}), \ldots, (\mathbf{q}_l^{(n)})$ (some of them possibly 
contain only finitely
many elements) such that any infinite sequence $(\mathbf{q}_j^{(n)})$ converges to the projective
point $(0:1:a_j)$. 
\end{proof}

\begin{example}
	Let $f(z_1,z_2)=z_1^2+z_2^2+1$. Then the zeros of $f_{h}$ at infinity are determined
by the equation $z_1^2+z_2^2=0$, giving the two zeros $(0:1:\pm i)$. 
Since the third coordinate is purely imaginary, any point on 
$\Sph^{1}$ is a limit direction, as already visualized
in the left picture of Figure \ref{fi:top1}.
\end{example}

\section{Open questions\label{se:openquestions}}

In this paper, we have introduced and developed the foundations of the imaginary projection of complex polynomial zero sets.
A central open question is
whether there exists an order map which distinguishes 
the different components of the complement, as in the case of amoebas. 
For coamoebas such an order map is known only in special cases so far 
(see \cite{forsgard-johansson-2015}). Moreover, no sharp upper bound 
(as a function of the underlying Newton polytope) is known for 
the number of components of the complement of an imaginary projection.

It is also an open problem to provide effective criteria for general $n$
to decide whether all points on the sphere are limit directions
of the imaginary projection of an $n$-variate polynomial.

\medskip

\noindent
{\bf Acknowledgment.} We would like to thank the referees for
helpful comments.

\bibliography{bibstable}
\bibliographystyle{plain}

\end{document}